\documentclass[11pt,a4paper]{amsart}
\usepackage[english]{babel}
\usepackage[T1]{fontenc}
\usepackage{mathpazo,amssymb,bm}
\usepackage{upref}
\usepackage{enumerate}
\usepackage{graphicx,xcolor}
\usepackage{cancel}
\usepackage{dsfont}       
\usepackage[colorlinks=true, pdftitle={A nonlinear problem associated
  with a subdifferential}, pdfauthor={Wolfgang Arendt, Daniel Hauer} ]{hyperref}

\title[Maximal $L^2$-regularity in nonlinear gradient systems]{Maximal
  $L^2$-regularity in nonlinear gradient systems and perturbations of
  sublinear growth}

\author{Wolfgang Arendt} 
\address[Wolfgang Arendt]{Institute of Applied Analysis, University of
  Ulm, 89069 Ulm, Germany}
\email{\href{mailto:wolfgang.arendt@uuni-ulm.de}{\nolinkurl{wolfgang.arendt@uni-ulm.de}}}

\author{Daniel Hauer} 
\address[Daniel Hauer]{School of Mathematics and
  Statistics, The University of Sydney, NSW 2006, Australia}
\email{\href{mailto:daniel.hauer@sydney.edu.au}{\nolinkurl{daniel.hauer@sydney.edu.au}}}

\thanks{The second author is very grateful for the warm hospitality
  received during his visits at the University of Ulm.}



\numberwithin{equation}{section}

\newtheorem{theorem}{Theorem}[section]
\newtheorem{proposition}[theorem]{Proposition}
\newtheorem{lemma}[theorem]{Lemma}

\theoremstyle{definition}
\newtheorem{definition}[theorem]{Definition}

\newtheorem{remark}[theorem]{Remark}
\newtheorem{example}[theorem]{Example}

\newcommand\R{{\mathbb{R}}}
\newcommand\N{\mathbb{N}}

\newcommand\T{\mathcal{T}}
\newcommand\He{\mathcal{H}}

\newcommand\dx{\mathrm{d}x }

\newcommand\ds{\mathrm{d}s }

\newcommand\dt{\mathrm{d}t }
\newcommand\td{\mathrm{d} }

\newcommand\abs[1]{\lvert#1\rvert}

\newcommand\norm[1]{\lVert#1\rVert}

\definecolor{darkred}{rgb}{0.7,0.1,0.1}

\begin{document}
\date{\today}
\maketitle

\tableofcontents

\begin{abstract}
  The nonlinear semigroup generated by the subdifferential of a convex
  lower semicontinuous function $\varphi$ has a smoothing effect,
  discovered by H. Br\'ezis, which implies maximal regularity for the
  evolution equation. We use this and Schaefer's fixed point theorem
  to solve the evolution equation perturbed by a Nemytskii-operator of
  sublinear growth. For this, we need that the sublevel sets of
  $\varphi$ are not only closed, but even compact. We apply our
  results to the $p$-Laplacian and also to the Dirichlet-to-Neumann
  operator with respect to $p$-harmonic functions.
\end{abstract}

%
%

\section{Introduction}
\label{sec:1}

Let $H$ be a real Hilbert space, $\varphi : H\to (-\infty,+\infty]$ a
proper, convex, lower semicontinuous function and 
$A=\partial\varphi$ the subdifferential of $\varphi$ (see
Section~\ref{sec:preliminary} for more details). Then $A$ is a
maximal monotone (in general, multi-valued) operator on $H$, for which
the following remarkable well-posedness result holds.

\begin{theorem}[{\bfseries Br\'ezis~\cite{MR0283635}}]\label{thm:Brezis-short} 
  Let $u_{0}\in H$ such that $\varphi(u_{0})$ is finite and $f\in
  L^{2}(0,T;H)$. Then, there exists a unique $u \in H^{1}(0,T;H)$ such that
  \begin{equation}
    \label{eq:1}
    \begin{cases}
      \dot{u}(t)+Au(t)\ni f(t) & \text{a.e. on $(0,T)$,}\\
      \hspace{39pt}u(0)=u_{0}.&
    \end{cases}
  \end{equation}
\end{theorem}

Our aim in this article is to study a perturbed version
of~\eqref{eq:1}. Let $\He$ denote the space $L^{2}(0,T;H)$, ($T>0$), and $G : \He\to
\He$ be a continuous mapping satisfying the \emph{sublinear} growth condition
\begin{equation}
  \label{eq:26}
  \norm{Gv(t)}_{H}\le L\,\norm{v(t)}_{H}+b(t)\qquad\text{a.e. on
    $(0,T)$ and for all $v\in \He$,}
\end{equation}
for some constants $L$, $b\in L^{2}(0,T)$ satisfying $b(t)\ge 0$ for
a.e. $t\in (0,T)$. Then we study the evolutionary problem
%
%
%
\begin{equation}
  \label{eq:2}
  \begin{cases}
    \dot{u}(t)+Au(t)\ni Gu(t) & \text{a.e. on $(0,T)$,}\\
    \hspace{39pt}u(0)=u_{0}. & 
  \end{cases}
\end{equation}

For that, we will use a compactness argument in form of Schaefer's
fixed point theorem (see Theorem~\ref{thm:1} in
Section~\ref{sec:preliminary}). Recall that lower semicontinuity of
$\varphi$ is equivalent to saying that the sublevel sets
$E_{c}:=\{u\in H\,\vert\,\varphi(u)\le c\}$, ($c\in \R$), are
closed. We will assume more, namely, compactness of sublevel sets
$E_{c}$. In fact, we need this assumption only for the shifted
function $\varphi_{\omega}$ given by
$\varphi_{\omega}(u)=\varphi(u)+\frac{\omega}{2}\norm{u}_{H}^{2}$
$(u\in H)$, which is important for applications. Then our main results
says the following.


\begin{theorem}
  \label{thm:main-short}
  Let $\varphi : H\to (-\infty,+\infty]$ be a proper function such
  that for some $\omega\ge 0$, $\varphi_{\omega}$ is convex and has
  compact sublevel sets. Let $A=\partial\varphi$ and $G : \He\to \He$
  be a continuous mapping satisfying~\eqref{eq:26}. Then for every
  $u_{0}\in H$ with $\varphi(u_{0})$ finite, there exists
  $u\in H^{1}(0,T;H)$ solving~\eqref{eq:2}.
\end{theorem}

We show in Example~\ref{ex:2} that the solution is not unique in
general. The proof of Theorem~\ref{thm:main-short} is based on
Br\'ezis' Theorem~\ref{thm:Brezis-short}. However, we need it under
the hypothesis that merely $\varphi_{\omega}$ is convex. We give a
proof of this more general result (see Theorem~\ref{thm:2}) in the appendix of this
paper. Theorem~\ref{thm:main-short} remains also true if $u_{0}\in
\overline{D(\varphi)}$ where $D(\varphi):=\{u\in
H\,\vert\,\varphi(u)<+\infty\}$; however, the solution of~\eqref{eq:2} is merely in
$H^{1}_{loc}((0,T];H)$ in that case.

As application, we consider $H=L^{2}(\Omega)$ and $G$ a Nemytskii
operator. The operator $A$ may be the $p$-Laplacian ($1\le p<+\infty$)
with possibly lower order terms and equipped with some boundary
conditions (Dirichlet, Neumann, or Robin, see~\cite{CoulHau2017}) or
a $p$-version of the Dirichlet-to-Neumann operator considered recently
in~\cite{MR3369257} and via the abstract theory of 
$j$-elliptic functions (see~\cite{MR2823661,MR2881534} and~\cite{MR3465809}).

%
%
%
%

\section{Preliminaries}
\label{sec:preliminary}
In this section, we define the precise setting used throughout this
paper and explain our mains tools: Brezis' result for semiconvex functions and Schaefer's fixed
point theorem.\medskip

We begin by recalling that a mapping $\T$ defined on a Banach space $X$
is called \emph{compact} if $\T$ maps bounded sets in into relatively
compact sets.

\begin{theorem}[{\cite{MR0342978}, {\bfseries Schaefer's fixed point
      theorem}}]
  \label{thm:1}
  Let $X$ be a Banach space and $\T : X \to X$ be continuous and
  compact. Assume that the ``Schaefer set''
  \begin{displaymath}
    \mathcal{S}:=\Big\{u\in X\,\Big\vert\,\text{there exists
    }\lambda\in [0,1]\text{ s.t. }u=\lambda \T{}u \Big\}
  \end{displaymath}
  is bounded in $X$. Then $\T$ has a fixed point.
\end{theorem}

This result is a special case of \emph{Leray-Schauder}'s fixed point
theorem, but Schaefer gave a most elegant proof (cf~\cite{MR2597943}),
which also is valid in locally compact spaces.\medskip

Given a function $\varphi : H \to (-\infty,+\infty]$, we call the set
$D(\varphi):=\{u\in H\,\vert\,\varphi(u)<+\infty\}$ the
\emph{effective domain} of $\varphi$, and $\varphi$ is said to be
\emph{proper} if $D(\varphi)$ is non-empty. Further, we say that
$\varphi$ is \emph{lower semicontinuous} if
for every $c\in \R$, the sublevel set
\begin{displaymath}
  E_{c}:=\Big\{u\in D(\varphi)\,\Big\vert\;\varphi(u)\le c\Big\}
\end{displaymath}
is closed in $H$, and $\varphi$ is \emph{semiconvex} if there exists
an $\omega\in \R$ such that the shifted function
$\varphi_{\omega} : H\to (-\infty,+\infty]$ defined by
\begin{displaymath}
\varphi_{\omega}(u):=\varphi(u)+\frac{\omega}{2}\norm{u}_{H}^{2},\qquad
(u\in H),
\end{displaymath}
is convex. Then, $\varphi_{\hat{\omega}}$ is convex for all
$\hat{w}\ge \omega$, and $\varphi_{\omega}$ is lower semicontinuous if
and only if $\varphi$ is lower semicontinuous.\medskip 

Given a function $\varphi : H\to (-\infty,+\infty]$, its \emph{subdifferential}
$A=\partial\varphi$ is defined by
\begin{displaymath}
  \partial\varphi=\Big\{(u,h)\in H\times
  H\,\Big\vert\,\liminf_{t\downarrow
    0}\frac{\varphi(u+tv)-\varphi(u)}{t}
  \ge (h, v)_{H}\,\forall\, v\in D(\varphi)\Big\},
\end{displaymath}
which, if $\varphi_{\omega}$ is convex, reduces to
\begin{displaymath}
  \partial\varphi=\Big\{(u,h)\in H\times
  H\,\Big\vert\,\varphi_{\omega}(u+v)-\varphi_{\omega}(u)\ge (h+\omega
  u, v)_{H}\,\forall\, v\in D(\varphi)\Big\}.
\end{displaymath}
It is standard to identify a (possibly multi-valued) operator $A$ on $H$
with its graph and for every $u\in H$, one sets
\begin{math}
   Au:=\{v\in H \,\vert\, (u,v)\in A\}
\end{math}   
and calls $D(A):=\{u\in H \,\vert\, Au\neq \emptyset\}$ the
\emph{domain of $A$} and
$\textrm{Rg}(A):=\bigcup_{u\in D(A)}\! Au$ the \emph{range of
  $A$}.\medskip

Now, suppose $\varphi : H\to (-\infty,+\infty]$ is proper, lower
semicontinuous, and semiconvex; more precisely, let's fix $\omega\in
\R$ such that $\varphi_{\omega}$ is convex. Then, under those
hypotheses on $\varphi$, Br\'ezis' well-posedness result
(Theorem~\ref{thm:Brezis-short}) remains true.

\begin{remark}[{\bfseries Maximal $L^2$-regularity}]
  If $u_{0}\in H$ such that $\varphi(u_{0})$ is finite, then
  Theorem~\ref{thm:Brezis-short} says that for every $f\in L^{2}(0,T;H)$, the unique solution
  $u$ of~\eqref{eq:1} has its time derivative $\dot{u}\in
  L^{2}(0,T;H)$ and hence by the differential inclusion
  \begin{equation}
    \label{eq:8}
    \dot{u}(t)+Au(t)\ni f(t)\qquad\text{a.e. on $(0,T)$,}
  \end{equation}
  also $Au\in L^{2}(0,T;H)$. In other words, for $f\in
  L^{2}(0,T;H)$, $\dot{u}$ and $Au\in L^{2}(0,T;H)$ admit the maximal
  possible regularity. 
  For this reason, we call this property \emph{maximal
    $L^{2}$-regularity}, as it is customary for generators of
  holomorphic semigroups on Hilbert spaces (see~\cite{MR2103696} for
  a survey on this subject). 
\end{remark}

As before, we fix $T>0$, denote by $\He$ the space $L^{2}(0,T;H)$, and write
$\norm{\cdot}_{\He}$ for the norm $\norm{\cdot}_{L^{2}(0,T;H)}$.\medskip

Further, after possibly replacing $\varphi$
by a translation, we may always assume without loss of generality that
$0\in D(\partial\varphi_{\omega})$ and $\varphi_{\omega}$ attains a
minimum at $0$ with $\varphi_{\omega}(0)=0$ (for further details
see~\cite[p159]{MR2582280} or the appendix of this paper). 
By the convexity of $\varphi_{\omega}$, this implies that 
$(0,0)\in \omega I_{H}+A$, that is,
\begin{equation}
  \label{eq:3}
  (h+\omega u,u)_{H}\ge 0\qquad\text{for all $(u,h)\in A$.}
\end{equation}

With this assumption in mind, we now state Br\'ezis' $L^2$-maximal
regularity theorem for semiconvex functions. 


\begin{theorem}[{{\bfseries Br\'ezis' $L^2$-maximal regularity for
      semiconvex $\varphi$}}]
  \label{thm:2}
  Let $u_{0}\in \overline{D(\varphi)}$ and $f\in \He$. Then, there exists a
  unique $u\in H^{1}_{loc}((0,T];H)\cap C([0,T];H)$ satisfying
  \begin{equation}
    \label{eq:5}
  \begin{cases}
    \dot{u}(t)+Au(t)\ni f(t) & \text{a.e. on $(0,T)$,}\\
    \hspace{39pt}u(0)=u_{0}.& 
  \end{cases}
\end{equation}
Moreover, one has that $\varphi\circ u\in W^{1,1}_{loc}((0,T])\cap L^{1}(0,T)$,
\begin{align}
  \label{eq:13}
  \norm{u(t)}_{H}&\le \left(
  \norm{u_{0}}_{H}^{2}+\int_{0}^{T}\norm{f(s)}_{H}^{2}\,\ds\right)^{\frac{1}{2}}\,e^{\frac{1+2\omega}{2} t}\;\text{
   for every $t\in (0,T]$,}\\ \label{eq:21}
  \int_{0}^{T}\varphi(u(s))\,\ds&\le
    \tfrac{1}{2}\norm{f}^{2}_{\He}+\tfrac{1+\omega}{2}
    \norm{u}^{2}_{\He} +\tfrac{1}{2}\norm{u_{0}}^{2}_{H},\\
  \label{eq:10}
   t\varphi(u(t))&\le
   \int_{0}^{T}\varphi(u(s))\,\ds+\tfrac{1}{2}\norm{\sqrt{\cdot}f}^{2}_{\He}\quad\text{
   for every $t\in (0,T]$,}\\
  \label{eq:11}
   \norm{\sqrt{\cdot}\dot{u}}^{2}_{\He}&\le 2
                                              \int_{0}^{T}\varphi(u(t))\,\dt+
                                              \norm{\sqrt{\cdot}f}^{2}_{\He}.
\end{align}
Finally, if $u_{0}\in D(\varphi)$, then $u\in H^{1}(0,T;H)$.
\end{theorem}

To keep this paper self-contained, we provide a  proof of this result
in the appendix of this paper.\medskip

\begin{definition}
  Given $f\in \He$ and $u_{0}\in H$, we call a function
  $u : [0,T]\to H$ a \emph{(strong) solution} of~\eqref{eq:5}
  (respectively, of~\eqref{eq:1}) if
  $u\in H^{1}_{loc}((0,T];H)\cap C([0,T];H)$, $u(0)=u_{0}$, and for
  a.e. $t\in (0,T)$, one has that $u(t)\in D(A)$ and
  $f(t)-\dot{u}(t)\in Au(t)$.
\end{definition}

For illustrating the theory developed in this paper, we consider the following
standard example: the \emph{Dirichlet $p$-Laplacian} perturbed by a lower order term.

\begin{example}\label{ex:1}
  Let $\Omega$ be an open subset of $\R^{d}$, ($d\ge 1$),
  $H=L^{2}(\Omega)$, and for $\frac{2d}{d+2}\le p<\infty$, let
  $V=W^{1,p}_{0}(\Omega)$ be the closure of $C_{c}^{1}(\Omega)$ with
  respect to the norm
  $\norm{u}_{V}:=\norm{\nabla u}_{L^{p}(\Omega;\R^{d})}$. Then, one
  has that $V$ is continuously embedded into $H$
  (cf~\cite[Theorem~9.16]{MR2759829}); we write for this $V
  \hookrightarrow H$. 

  Further, let $f=\beta+f_{1}$ be
  the sum of a maximal monotone graph $\beta$ of
  $\R$ satisfying $(0,0)\in \beta$ and a \emph{Lipschitz-Carath\'eodory function}
  $f_{1} : \Omega\times \R\to \R$ satisfying $f(x,0)=0$; that is, for a.e.
  $x\in \Omega$, $f_{1}(x,\cdot)$ be Lipschitz continuous (with
  constant $\omega>0$) uniformly for a.e. $x\in \Omega$, and
  $f_{1}(\cdot,u)$ is measurable on $\Omega$ for every $u\in
  \R$. Then, there is a proper, convex and lower semicontinuous function
  $j:\R\to (-\infty,+\infty]$ satisfying $j(0)=0$ and $\partial
  j=\beta$ in $\R$ (see~\cite[Example 1., p53]{MR2582280}). We set
  \begin{align}\notag
    F(u)&=\phi(u)+\int_{\Omega}F_{1}(u(x))\,dx\qquad\text{ for
          every $u\in H$, where}\\ \label{eq:22}
    \phi(u)&=
          \begin{cases}
            \displaystyle\int_{\Omega} j(u(x))\,\dx &\qquad\text{if $j(u)\in
              L^{1}(\Omega)$,}\\[7pt]
            +\infty &\qquad\text{if otherwise, and}
          \end{cases}\\ \notag
    F_{1}(u) &=\int_{0}^{u(x)}f_{1}(\cdot,s)\,\ds
  \end{align}
  for every $u\in L^{2}(\Omega)$. Now, let $\varphi_{1} : H \to (-\infty,+\infty]$ be given by
  \begin{displaymath}
    \varphi_{1}(u)=
    \begin{cases}
      \displaystyle\tfrac{1}{p}\int_{\Omega}\abs{\nabla u}^{p}\,\dx+\int_{\Omega}F_{1}(u)\,\dx &
      \text{if $u\in V$,}\\[7pt]
      +\infty & \text{if $u\in H\setminus V$}
    \end{cases}
  \end{displaymath}
  for every $u\in H$. Then the domain $D(\varphi_{1})$ of
  $\varphi_{1}$ is $V$. The function $\varphi_{1}$ is lower semicontinuous on $H$,
  proper, $\varphi_{1\!,\omega}$ is convex, and for every $u\in
  V$, $\varphi_{1}$ is G\^ateaux-differentiable with
  \begin{displaymath}
    D_{v}\varphi(u)=\lim_{t\to0+}\frac{\varphi(u+t)-\varphi(u)}{t}=\int_{\Omega}\abs{\nabla
    u}^{p-2}\nabla u\nabla v+f_{1}(x,u)\,v\dx
  \end{displaymath}
  for every $v\in V$. Since $V$ is dense in $H$, the operator
  $\partial\varphi_{1}$ is a single-valued operator on $H$ with domain
  \begin{align*}
    D(\partial\varphi_{1})&=\Big\{u\in V\,\Big\vert\,\exists\,h\in
    L^{2}(\Omega)\,D_{\nu}\varphi(u)=\int_{\Omega}h v\,\dx\,\forall\,v\in
    V\Big\},\text{ and}\\[7pt]
        \partial\varphi_{1}(u) &=h=-\Delta_{p}u+f_{1}(x,u)\qquad\text{in $\mathcal{D}'(\Omega)$.}
  \end{align*}
  The operator $\partial\varphi_{1}$ is the negative \emph{Dirichlet
    $p$-Laplacian $-\Delta_{p}^{\! D}$ on $\Omega$ with a Lipschitz continuous lower order
    term} $f_{1}$. Next, we add the function $\phi$ given by~\eqref{eq:22}
  to the $\varphi_{1}$. For this, note that $\phi$ is proper (since
  for $u_{0}\equiv 0$, $\phi(u_{0})=0$) with
  $\textrm{int}(D(\phi))\neq \emptyset$, convex (since $j$ is convex),
  and lower semicontinuous on $H$. Thus, the function
  $\varphi : H\to (-\infty,+\infty]$ given by
  \begin{equation}
    \label{eq:23}
    \varphi(u)=\varphi_{1}(u)+\phi(u)\qquad\text{for every $u\in H$,}
  \end{equation}
  is convex, lower semicontinuous, and proper with domain
  $D(\varphi)=\{u\in V\,\vert\,j(u)\in L^{1}(\Omega)\}$ and 
  the operator $A=\partial\varphi$ is given by
  \begin{align*}
    D(A)&=\Big\{u\in D(\varphi)\,\Big\vert\,\exists\,h\in
    L^{2}(\Omega)\,D_{\nu}\varphi(u)=\int_{\Omega}h v\,\dx\,\forall\,v\in
    D(\varphi)\Big\},\\
          Au&=h=-\Delta_{p}u+\beta(u)+f_{1}(x,u),
  \end{align*}
  and $A$ is single-valued provided $D(\varphi)$ is dense in
  $L^{2}(\Omega)$. Here, we note that
  \begin{displaymath}
    \overline{D(A)}=\overline{D(\varphi)}=\Big\{u\in H\,\Big\vert\,j(u(x))\in
    \overline{D(\beta)}\textrm{ for a.e. $x\in \Omega$}\Big\}.
  \end{displaymath}
  Due to Theorem~\ref{thm:2}, for every $u_{0}\in
  \overline{D(\varphi)}$ and $f\in \He$, there is a unique
  solution $u\in H^{1}_{loc}((0,T];H)\cap C([0,T];H)$ of the parabolic
  boundary-value problem
  \begin{displaymath}
    \begin{cases}
      \partial_{t}u(t)
      -\Delta_{p}u(t)+\beta(u(t))+f_{1}(\cdot,u(t))\ni f(t) & \text{on
      $(0,T)\times \Omega$,}\\
    \hspace{1,3cm}\phantom{-\Delta_{p}u(t)+\beta(u(t))+f_{1}(\cdot)}u(t)=0 & \text{on
      $(0,T)\times \partial\Omega$,}\\
    \hspace{1,25cm}\phantom{-\Delta_{p}^{\! D}u(t)+\beta(u(t))+f_{1}(\cdot)}u(0)=u_{0} &\text{on
      $\Omega$.}
    \end{cases}
  \end{displaymath}
  Here, we write $\partial_{t}u(t)$ instead of $\dot{u}(t)$ since we
  rewrote the abstract Cauchy problem~\eqref{eq:5} as an explicit parabolic
  partial differential equation. 
\end{example}

%
%

\section{Main result}
\label{sec:main-results}

Throughout this section, let $\varphi : H\to (-\infty,+\infty]$ be a
proper function. We assume that there is an $\omega\in \R$ such that
$\varphi_{\omega}$ is convex and the sublevel set
\begin{equation}
  \label{eq:4}
  E_{\omega\!,c}:=\Big\{u\in D(\varphi)\,\vert \varphi_{\omega}(u)\le
    c\Big\}\quad\text{is compact in $H$ for every $c\in \R$.}
\end{equation}
\begin{flushright} 
  $\mbox{}_{\Box}$
\end{flushright}

\begin{remark}
  We emphasize that condition~\eqref{eq:4} does not imply that
  $\varphi$ has compact sublevel sets. This becomes more clear if one
  considers as $\varphi$ the function associated with the negative \emph{Neumann
    $p$-Laplacian} $-\Delta_{p}^{\!\! N}$ on a
  boun\-ded, open subset $\Omega$ of $\R^{d}$ with a Lipschitz boundary
  $\partial\Omega$. For $\max\{1,\frac{2d}{d+2}\}<p<\infty$, ($d\ge
  1$), let $V=W^{1,p}(\Omega)$, $H=L^{2}(\Omega)$,
  $\He=L^{2}(0,T;L^{2}(\Omega))=L^{2}((0,T)\times \Omega)$, and
  $\varphi : H\to (-\infty,+\infty]$ be given by  
  \begin{equation}
    \label{eq:24}
    \varphi(u):=
    \begin{cases}
      \tfrac{1}{p}\displaystyle\int_{\Omega}\abs{\nabla u}^{p}\dx & \text{if $u\in
        V$,}\\[7pt]
      +\infty & \text{if $u\in H\setminus V$}
    \end{cases}
\end{equation}
 for every $u\in H$. Then, for every $c>0$, the sublevel set
  $E_{0,c}$ of $\varphi$ contains the sequence
  $(u_{n})_{n\ge 0}$ of constant functions $u_{n}\equiv n$, which does
  not admit any convergent subsequence in $H$. On the
  other hand, for every $\omega>0$ and $c>0$, the sublevel set
  $E_{\omega\!,c}$ is a bounded set in $V$ and by Rellich-Kandrachov's
  compactness, one has that $V\hookrightarrow H$ by a compact
  embedding. Thus, for every $\omega>0$ and $c>0$, the sublevel set
  $E_{\omega\!,c}$ is compact in $L^{2}(\Omega)$.
\end{remark}

Let $G : \He\to \He$ be a continuous function with \emph{sublinear
growth}; that is, there are $L\ge 0$ and $b\in L^{2}(0,T)$ satisfying
$b(t)\ge 0$ a.e. on $(0,T)$ such that
\begin{equation}
  \tag{\ref{eq:26}}
  \norm{Gv(t)}_{H}\le L\,\norm{v(t)}_{H}+b(t)\qquad\text{a.e. on
    $(0,T)$, for all $v\in \He$.}
\end{equation}
Here we let $Gv(t):=(Gv)(t)$ to use less heavy notation. 
Then, our main result of this paper reads as follows.

\begin{theorem}
  \label{thm:3}
  Let $u_{0}\in \overline{D(\varphi)}$ and $f\in \He$. Then, there
  exists a solution $u\in H^{1}_{loc}((0,T];H)\cap C([0,T];H)$ of
  \begin{equation}
    \label{eq:6}
  \begin{cases}
    \dot{u}(t)+Au(t)\ni Gu(t) & \text{a.e. on $(0,T)$,}\\
    \hspace{39pt}u(0)=u_{0}.& 
  \end{cases}
\end{equation}
In particular, if $u_{0}\in D(\varphi)$, then problem~\eqref{eq:6} has
a solution $u\in H^{1}(0,T;H)$.
\end{theorem}

Note that $Gu\in \He$. Thus, the inclusion in~\eqref{eq:6} means
that $Gu(t)-\dot{u}(t)\in Au(t)$ a.e. on $(0,T)$. In particular, the
following regularity estimates hold for strong solutions
of~\eqref{eq:6}.

\begin{remark}
  For given $u_{0}\in\overline{D(\varphi)}$ and $f\in \He$, the
  solution $u$ of~\eqref{eq:6} satisfies
  \begin{displaymath}
    \varphi\circ u\in W^{1,1}_{loc}((0,T])\cap L^{1}(0,T),
  \end{displaymath}
\begin{equation} \label{eq:14}
   \norm{u(t)}_{H}\le \left(\norm{u_{0}}_{H}^{2}+
    \norm{b}_{L^{2}(0,T)}^{2}\right)^{\frac{1}{2}}\,e^{\frac{2L+1+2\omega}{2}\,t}
  \quad\text{for all $t\in [0,T]$.}
\end{equation}
{\hfill$\mbox{}_{\Box}$}
\end{remark}



The main example of perturbations $G$ allowed in Theorem~\ref{thm:3}
are Nemytskii operators on $\He=L^{2}(0,T;L^{2}(\Omega))$. Let
$\Omega\subseteq \R^{d}$ be open and
$g : (0,T)\times \Omega\times \R\to \R$ be a \emph{Carath\'eodory
  function}, that is,
\begin{align*}
\bullet\qquad  &g(\cdot,\cdot,v) : (0,T)\times \Omega\to \R\quad\text{is
    measurable, for all $v\in \R$,}\\
\bullet\qquad  &g(t,x,\cdot) : \R\to \R\quad\text{is
    continuous, for a.e. $(t,x)\in (0,T)\times \Omega$.}
\end{align*}
Assume furthermore that $g$ has \emph{sublinear growth}, that is,
there exist $L\ge 0$ and $f\in L^{2}(0,T;L^{2}(\Omega))$ such that
\begin{equation}
  \label{eq:27}
  \abs{g(t,x,v)}\le L\,\abs{v}+f(t,x)\quad\text{for all $v\in \R$,
    a.e. $(t,x)\in  (0,T)\times \Omega$.}
\end{equation}

\begin{proposition}\label{prop:3}
  Let $\He=L^{2}(0,T;L^{2}(\Omega))$. Then, the relation
  \begin{equation}
    \label{eq:28}
    Gv(t,x):=g(t,x,v(t,x))\quad\text{for a.e. $(t,x)\in
      (0,T)\times \Omega$, and every $v\in \He$,}
  \end{equation}
  defines a continuous operator $G : \He\to\He$ of
  sublinear growth~\eqref{eq:26}.
\end{proposition}

The proof is routine (cf~\cite[Proposition~26.7]{MR1033498}) if one
uses that $f_{n}\to f$ in $\He$ if and only if each subsequence of
$(f_{n})_{n\ge 1}$ has a dominated subsequence converging to $f$
a.e. (which is well known from the completeness proof of
$L^2$).\medskip


We illustrate our result by reconsidering Example~\ref{ex:1} adding a
perturbation of Nemytskii type.

\begin{example}[{\bfseries Example~\ref{ex:1} revisited}]\label{ex:1bis}
  For $\max\{1,\frac{2d}{d+2}\}<p<\infty$, let
  $V=W^{1,p}_{0}(\Omega)$, $H=L^{2}(\Omega)$,
  $\He=L^{2}((0,T)\times\Omega)$ and let $\varphi$ be given
  by~\eqref{eq:23}.  Then, there is an $\omega>0$ such that
  $\varphi_{\omega}$ is convex and for every $c>0$, the sublevel set
  $E_{\omega\!,c}$ is compact in $L^{2}(\Omega)$. Furthermore, let
  $g : (0,T)\times \Omega\times \R\to \R$ be a Carath\'edory function
  with sublinear growth and $u_{0}\in \overline{D(\varphi)}$. Then,
  there is at least one solution
  $u\in H^{1}_{loc}((0,T];H)\cap C([0,T];H)$ of the parabolic
  boundary-value problem
  \begin{displaymath}
    \begin{cases}
      \partial_{t}u(t,\cdot)
      -\Delta_{p}u(t,\cdot)+\beta(u(t,\cdot))+f_{1}(\cdot,u(t,\cdot))\ni g(t,\cdot,u(t,\cdot)) & \text{on
      $(0,T)\times \Omega$,}\\
    \hspace{6.56cm}u(t,\cdot)=0 & \text{on
      $(0,T)\times \partial\Omega$,}\\
    \hspace{6.52cm}u(0,\cdot)=u_{0} &\text{on
      $\Omega$.}
    \end{cases}
  \end{displaymath}
\end{example}

In general, the solutions in Example~\ref{ex:1bis} are not unique. We
give an example.

\begin{example}[{\bfseries Non-uniqueness}]\label{ex:2}
  Let $g(u)=\sqrt{\abs{u}}$, $u\in \R$, and $\Omega$ be an open
  and bounded subset of $\R^{d}$, $d\ge 1$, with a Lipschitz boundary
  $\partial\Omega$. Then, there are $L$, $b>0$ such that $\hat{g}$ satisfies
  \begin{displaymath}
    \abs{g(u)}\le L\,\abs{u}+b\qquad\text{for every $u\in \R$.}
  \end{displaymath}
  Thus, for $H=L^{2}(\Omega)$, one has that
  $\He=L^{2}((0,T)\times\Omega)$ and the associated Nemytskii operator
  $G : \He\to \He$ defined by~\eqref{eq:28} satisfies the sublinear growth condition~\eqref{eq:26}. For
  $\max\{1,\frac{2d}{d+2}\}<p<+\infty$, let
  $\varphi : L^{2}(\Omega)\to (-\infty,+\infty]$ be the energy
  function~\eqref{eq:24} associated with the negative Neumann
  $p$-Laplacian $-\Delta_{p}^{\!  N}$ on $\Omega$. Then, by Theorem~\ref{thm:3},
  for every $u_{0}\in L^{2}(\Omega)$ and every $T>0$, there is a
  solution $u \in H^{1}_{loc}((0,T]; L^{2}(\Omega))\cap C([0,T];L^{2}(\Omega))$ of
  \begin{equation}
    \label{eq:15}
  \begin{cases}
   \hspace{10pt}\partial_{t}u(t,\cdot)-\Delta_{p}^{\! N}u(t,\cdot)= \sqrt{\abs{u}(t,\cdot)}&
   \text{in $(0,T)\times \Omega$,}\\
    \abs{\nabla u(t,\cdot)}^{p-2}D_{\nu}u(t,\cdot)=0 &\text{on
     $(0,T)\times \partial\Omega$,}\\
    \hspace{80pt}u(0)=u_{0}&\text{on $\Omega$.}\\
  \end{cases}
\end{equation}
Here, $\abs{\nabla u}^{p-2}D_{\nu}u$ denotes the (weak) co-normal
derivative of $u$ on $\partial\Omega$ (cf~\cite{CoulHau2017}). Now, for the initial value
$u_{0}\equiv 0$ on $\Omega$, the constant zero function $u\equiv 0$ is
certainly a solution of~\eqref{eq:15}. For constructing a non-trivial
solution of~\eqref{eq:15} with initial value $u_{0}\equiv 0$, let
$w\in C^{1}[0,T]$ be a non-trivial solution of the following classical
ordinary differential equation
\begin{equation}
  \label{eq:16}
  w'=\sqrt{\abs{w}}\text{ on $(0,T)$, $w(0)=0$,}
\end{equation}
For instance, one non-trivial solution is $w(t)=t^2/4$. Since for
every constant $c\in \R$, $-\Delta_{p}^{\!
  N}(c\mathds{1}_{\Omega})=0$, the function $u(t):=w(t)$ is another
non-trivial solution of~\eqref{eq:15} with initial value $u_{0}\equiv 0$. 
\end{example}

%
%
\section{Proof of the main result}
\label{sec:proof}

For the proof of Theorem~\ref{thm:3}, we need some auxiliary
results. The first concerns continuity and is standard (see
B\'enilan~\cite[(6.5), p87]{Benilanbook} or Barbu~\cite[(4.2),
p128]{MR2582280}).

\begin{lemma}
  \label{lem:1}
  Let $f_{1}$, $f_{2}\in \He$, $u_{1}$, $u_{2}\in
  H^1(0,T;H)$ such that
  \begin{align*}
    & \dot{u}_{1}+Au_{1}\ni f_{1}\qquad\text{on $(0,T)$,}\\
    & \dot{u}_{2}+Au_{2}\ni f_{2}\qquad\text{on $(0,T)$.}
  \end{align*}
  Then,
  \begin{equation}
    \label{eq:9}
    \norm{u_{1}(t)-u_{2}(t)}_{H}\le e^{\omega
      t}\norm{u_{1}(0)-u_{2}(0)}_{H}+\int_{0}^{t}e^{\omega (t-s)}\norm{f_{1}(s)-f_{2}(s)}_{H}\,\ds
  \end{equation}
  for every $t\in [0,T]$.
\end{lemma}

Next, we establish the compactness of the \emph{solution operator} $P$
associated with evolution problem~\eqref{eq:5}. Note, for convenience,
we write here $\He$ to denote $\He$, ($T>0$), and
recall that the closure $\overline{D(\varphi)}$ in $H$ of the effective
domain of a semiconvex function $\varphi$ is a convex subset of $H$.

\begin{lemma}
  \label{lem:2bis}
  Let $P : \overline{D(\varphi)}\times \He\to \He$ be the mapping defined dy
  \begin{displaymath}
    P(u_{0},f)=\textit{``solution $u$
      of~\eqref{eq:5}''}\qquad\text{for every $u_{0}\in \overline{D(\varphi)}$ and $f\in \He$.}
  \end{displaymath}
  Then, $P$ is continuous and compact.
\end{lemma}

\begin{proof}
  (a) By Lemma~\ref{lem:1}, the map $P$ is continuous from
  $\overline{D(\varphi)}\times \He$ to $\He$.\medskip

  (b) We show that $P$ is compact. Let
  $(u_{n}^{(0)})_{n\ge 1}\subseteq \overline{D(\varphi)}$ and
  $(f_{n})_{n\ge 1}\subseteq \He$ such that
  $\norm{u_{n}^{(0)}}_{H}+\norm{f_{n}}_{\He}\le c$ and
  $u_{n}=P(u_{n}^{(0)},f_{n})$ for every $n\ge 1$. Then, by~\eqref{eq:13},
  \eqref{eq:21} and by~\eqref{eq:11}, for every $\delta\in (0,T)$,
  there is a $c_{\delta}>0$ such that
  \begin{displaymath}
    \sup_{n\ge 1}\norm{u_{n}}_{H^{1}(\delta,T;H)}\le c_{\delta}.
  \end{displaymath}
  Since $H^{1}(\delta,T;H)\hookrightarrow C^{\gamma}([\delta,T];H)$
  for some $\gamma\in (0,1)$, it follows that the sequence
  $(u_{n})_{n\ge 1}$ is equicontinuous on $[\delta,T]$ for each
  $0<\delta<T$. Choose a countable dense subset $D:=\{t_{m}\vert\,m\in
  \N\}$ of $(0,T]$. Let $m\ge 1$. Then by~\eqref{eq:10},
  \begin{displaymath}
    \sup_{n\ge 1}\varphi(u_{n}(t_{m}))\qquad\text{ is finite}
  \end{displaymath}
  and since by~\eqref{eq:13}, $(u_{n}(t_{m}))_{n\ge 1}$ is bounded in $H$,
  there is a $c'>0$ such that $(u_{n}(t_{m}))_{n\ge 1}$ is in the
  sublevel set $E_{\!\omega,c'}$. Thus and by the
  assumption~\eqref{eq:4}, $(u_{n}(t_{m}))_{n\ge 1}$ has a convergent
  subsequence in $H$. By Cantor's diagonalization argument, we find a
  subsequence $(u_{n_{k}})_{k\ge1}$ of $(u_{n})_{n\ge 1}$ such that
  \begin{displaymath}
    \lim_{k\to +\infty}u_{n_{k}}(t_{m})\qquad\text{exists in $H$ for
      all $m\in \N$.}
  \end{displaymath}
  It follows from the equicontinuity of $(u_{n_{k}})_{k\ge 1}$ that
  $u_{n_{k}}$ converges in $C([\delta,T];H)$ for all $\delta\in
  (0,T]$. In particular, $(u_{n_{k}}(t))_{k\ge 1}$ converges in $H$
  for every $t\in (0,T)$ and by~\eqref{eq:13},  $(u_{n_{k}})_{k\ge 1}$
  is uniformly bounded in $L^{\infty}(0,T;H)$. Thus, it follows from
  Lebesgue's dominated convergence theorem that
  $u_{n_{k}}=P(u_{n_{k}}^{(0)},f_{n_{k}})$ converges in $\He$.
\end{proof}

\begin{remark}
  In the previous proof, we have actually shown that $P$ is compact
  from $\overline{D(\varphi)}\times \He$ into the
  Fr\'echet space $C((0,T];H)$.
\end{remark}

With these preliminaries, we can now give the proof of our main
result. Here, we got inspired from the linear case (cf~\cite{MR2722654}).

\begin{proof}[Proof of Theorem~\ref{thm:3}]
  First, let $u_{0}\in \overline{D(\varphi)}$.\medskip

  Let $v\in \He$. Then $Gv\in \mathcal{H}$ and so, by
  Br\'ezis' maximal $L^{2}$-regularity result (Theorem~\ref{thm:2}),
  there is a unique solution
  $u\in H^{1}_{loc}((0,T];H)\cap C([0,T];H)$ of the evolution problem
  \begin{displaymath}
  \begin{cases}
    \dot{u}(t)+Au(t)\ni Gv(t) & \text{a.e. on $(0,T)$,}\\
    \hspace{39pt}u(0)=u_{0}.& 
  \end{cases}
\end{displaymath}
Let $\T{}v:=P(u_{0},Gv)$. Then by the continuity of $G$ and since
$P(u_{0},\cdot) : \He \to\He$ is continuous and compact
(Lemma~\ref{lem:2bis}), the mapping $\T : \He\to \He$
is continuous and compact.\medskip

a) We consider the Schaefer set  
\begin{displaymath}
    \mathcal{S}:=\Big\{u\in \He\,\Big\vert\,\text{there exists
    }\lambda\in [0,1]\text{ s.t. }u=\lambda \T{}u \Big\}.
  \end{displaymath}
We show that $\mathcal{S}$ is bounded in $\mathcal{H}$. Let $u\in
\mathcal{S}$. We may assume that $\lambda\in (0,1]$, otherwise,
$u\equiv 0$. Then, one has that $u\in H^{1}_{loc}((0,T];H)\cap C([0,T];H)$ and 
\begin{displaymath}
  \begin{cases}
  \displaystyle  \frac{\dot{u}}{\lambda}
  +A\left(\frac{u}{\lambda}\right)\ni Gu & \text{on $(0,T)$,}\\
    \hspace{1,2cm}u(0)=u_{0}.& 
  \end{cases}
\end{displaymath}
It follows from~\eqref{eq:3} that
\begin{displaymath}
  \left(-\frac{\dot{u}}{\lambda}(t)+Gu(t)+\omega
  \frac{u}{\lambda}(t), \frac{u}{\lambda}\right)_{\!\! H}\ge 0\qquad\text{for
    a.e. $t\in (0,T)$.}
\end{displaymath}
Thus and by~\eqref{eq:26}, 
\begin{align*}
  \frac{\td}{\dt}\tfrac{1}{2}\norm{u(t)}_{H}^{2}
  & = (\dot{u}(t),u(t))_{H}\\
  & = (\dot{u}(t)-\lambda Gu(t)-\omega \lambda u(t),u(t))_{H}\\
  &\hspace{3cm}  +(\lambda Gu(t)+\omega \lambda u(t),u(t))_{H}\\
  & \le (\lambda Gu(t)+\omega \lambda u(t),u(t))_{H}\\
  & \le \lambda \left(
    \norm{Gu(t)}_{H}\,\norm{u(t)}_{H}+\omega
    \,\norm{u(t)}_{H}^{2}\right)\\
   & \le \lambda \left(
      L\,\norm{u(t)}_{H}^{2}+b(t)\,\norm{u(t)}_{H}+\omega \,\norm{u(t)}_{H}^{2}\right)\\
   & \le (2L+1+2\omega)\, \tfrac{1}{2}\norm{u(t)}_{H}^{2}  +\tfrac{1}{2}b^{2}(t)
\end{align*}
for a.e. $t\in (0,T)$. 
It follows from Gronwall's lemma that~\eqref{eq:14} holds for every
$t\in [0,T]$. Thus, $\mathcal{S}$ is bounded in $\mathcal{H}$. Now,
Schaefer's fixed point theorem implies that there exists
$u\in \mathcal{H}$ such that $u=\T{}u$; that is, 
$u\in H^{1}_{loc}((0,T];H)\cap C([0,T];H)$ is a solution of the evolution
problem~\eqref{eq:6}.\medskip 

b) Let $u_{0}\in D(\varphi)$. Then, by the first part of this proof,
there is a solution solution $u\in H^{1}_{loc}((0,T];H)\cap C([0,T];H)$ of the evolution
problem~\eqref{eq:6}. However, by Br\'ezis' maximal regularity result
applied to $f=Gu\in \He$, it follows that $u\in
H^{1}(0,T;H)$. This completes the proof of this theorem.
\end{proof}

%
%
\section{Application to $j$-elliptic functions}
\label{sec:j-elliptic-functionals}

In the previous examples (cf~Examples~\ref{ex:1} and
Example~\ref{ex:2}), $V$ is a Banach space injected in $H$. 
Recently, in 
\cite{MR3465809}, Chill, Hauer and Kennedy extended results of
\cite{MR2823661}, \cite{MR2881534} by Arendt and Ter Elst to a nonlinear framework of
\emph{$j$-elliptic functions} $\varphi : V\to (-\infty,+\infty]$
generating a quasi maximal monotone operator $\partial_{j}\varphi$ on
$H$, where $j : V\to H$ is just a linear operator which is not
necessarily injective. This enabled the authors of~\cite{MR3465809} to show that several
coupled parabolic-elliptic systems can be realized as a gradient
system in a Hilbert space $H$ and to extend the linear variational
theory of the Dirichlet-to-Neumann operator to the nonlinear
$p$-Laplace operator (see also~\cite{MR3403408,arXiv:1804.08272,HauMaz2019} for
further applications and extensions of this theory).\medskip

The aim of this section is to illustrate that the main
Theorem~\ref{thm:3} of Section~\ref{sec:main-results} can also be
applied to the framework of $j$-elliptic functions.\medskip

Let us briefly recall some basic notions and facts about $j$-elliptic
functions from~\cite{MR3465809}. Let $V$ be a real locally convex
topological vector space and $j: V \to H$ be a linear operator
which is merely weak-to-weak continuous (and, in general, not
injective). Given a function $\varphi : V \to  (-\infty,+\infty]$,
then the \emph{$j$-subdifferential} is the operator
\begin{displaymath}
  \partial_j\varphi := \Bigg\{ (u,f)\in H\times H\;\Bigg\vert\;
  \begin{array}[c]{c}
    \exists \hat{u}\in D(\varphi )
    \text{ s.t. } j(\hat{u}) = u \text{ and for every } \hat{v}\in V, \\[1pt]
   \displaystyle \liminf_{t\searrow 0} \frac{\varphi(\hat{u} +
      t\hat{v}) -\varphi(\hat{u})}{t} \ge (f,j(\hat{v}))_{H}
  \end{array}
  \Bigg\}.
\end{displaymath}
The function $\varphi$ is called \emph{$j$-semiconvex} if there exists
$\omega\in\R$ such that the ``shifted'' function $\varphi_\omega : V
\to (-\infty,+\infty]$ given by
\begin{displaymath}
  \varphi(\hat{u}) + \frac{\omega}{2} \,
  \norm{j(\hat{u})}_H^2\qquad\text{for every $\hat{u}\in V$,}
\end{displaymath}
is convex. If $V=H$ and $j=I_{H}$, then $j$-semiconvex functions
$\varphi$ are the \emph{semiconvex} ones (see
Section~\ref{sec:1}). The function $\varphi$ is called
\emph{$j$-elliptic} if there exists $\omega\ge 0$ such that
$\varphi_\omega$ is convex and for every $c\in \R$, the sublevel sets
$\{\hat{u}\in V\,\vert\, \varphi_{\omega}(u) \le c\}$ are relatively weakly
compact. Finally, we say that the function $\varphi$ is {\em lower
  semicontinuous} if the sublevel sets $\{ \varphi \le c\}$ are closed in
the topology of $V$ for every $c\in\R$. It was highlighted
in~\cite[Lemma~2.2]{MR3465809} that
\begin{enumerate}[(a)]
\item If $\varphi$ is $j$-semiconvex, then there is an $\omega\in
  \R$ such that 
    \begin{displaymath}
      \partial_j\varphi  = \Bigg\{ (u,f)\in H\times H \;\Bigg\vert\;
      \begin{array}[c]{c}
        \exists \hat{u}\in D(\varphi)
        \text{ s.t. }
        j(\hat{u}) = u \text{ and for every } \hat{v}\in V \\[0,1cm]
        \varphi_\omega (\hat{u} + \hat{v}) - \varphi_\omega
        (\hat{u}) \ge (f + \omega j(\hat{u}) , j(\hat{v}))_{H}
      \end{array}
      \Bigg\} .
    \end{displaymath}
  \item[(b)] If $\varphi$ is G\^ateaux differentiable with directional
    derivative $D_{\hat{v}}\varphi$, ($\hat{v}\in V$), then
    \begin{displaymath}
      \partial_j\varphi  = \Bigg\{ (u,f)\in H\times H \;\Bigg\vert\;
      \begin{array}[c]{c}
        \exists \hat{u}\in D(\varphi)
        \text{ s.t. } j(\hat{u}) = u \text{ and for every }
        \hat{v}\in V\\[0,1cm]
        D_{\hat{v}}\varphi(\hat{u})= (f,j(\hat{v}))_H
      \end{array}
      \Bigg\} .
    \end{displaymath}
\end{enumerate}


The main result in~\cite{MR3465809} is that the $j$-subdifferential $\partial_{j}\varphi$
of a $j$-elliptic function $\varphi$ is already a
classical subdifferential. More precisely, the following holds.

\begin{theorem}[{\cite[Corollary~2.7]{MR3465809}}]\label{cor:27}
  Let $\varphi : V\to (-\infty,+\infty]$ be proper, lower semicontinuous, 
  and $j$-elliptic. Then there is a proper, lower semicontinuous,
  semiconvex function $\varphi^{H} : H\to (-\infty,+\infty]$ such that
  $\partial_j\varphi= \partial\varphi^{H}$. The function $\varphi^{H}$
  is unique up to an additive constant.
\end{theorem}

Thus the operator $A=\partial_j\varphi$ has the properties of maximal
regularity we used before. The following result gives a description of
$\varphi^{H}$ in the convex case and will be important for our
intentions in this paper.



\begin{theorem}[{\cite[Theorem~2.9]{MR3465809}}]\label{thm:29}
  Assume that $\varphi : V\to (-\infty,+\infty]$ is convex, proper,
  lower semicontinuous and $j$-elliptic, and let
  $\varphi^{H} : H\to (-\infty,+\infty]$ be the function from
  Corollary~\ref{cor:27}. Then, there is a constant $c\in \R$ such
  that
  \begin{displaymath}
    \varphi^{H}(u) = c+\inf_{\hat{u}\in
      j^{-1}(\{u\})}\varphi(\hat{u})\qquad\text{for every $u\in H$}
  \end{displaymath}
  with effective domain $D(\varphi^{H}) = j(D(\varphi))$.
\end{theorem}


For our perturbation result, we need the compactness of the sublevel
sets of $\varphi^{H}$. With the help of Theorem~\ref{thm:29} we can
establish a criterion in terms of the given $\varphi$ for this
property.

\begin{lemma}\label{lem:2}
  Let $\varphi : V\to (-\infty,+\infty]$ be proper, lower
  semicontinuous $j$-semiconvex, and $j$-elliptic. Assume that
  \begin{equation}
    \label{eq:25}
    \begin{cases}
      &\text{$j : V\to H$ maps weakly relatively compact sets of $V$}\\
      & \text{into relatively norm-compact sets of $H$,}
    \end{cases}
  \end{equation}
  then there is an $\omega\ge 0$ such that for every $c\in \R$, the sublevel set
  \begin{displaymath}
    E_{\omega\!,c}
    =\Big\{u\in H\,\Big\vert\;\varphi^{H}_{\omega}(u)\le
    c\Big\}\qquad\text{is compact in $H$.}
  \end{displaymath}
\end{lemma}

\begin{remark}
  If $V$ is a normed space, then by the Eberlein-\v Smulian Theorem hypothesis~\eqref{eq:25} is
  equivalent to \emph{$j$ maps weakly convergent sequences in $V$ to norm
      convergent sequences in $H$}. This in turn is equivalent to $j$ being compact if $V$ is reflexive.
  
\end{remark}

\begin{proof}[Proof of Lemma~\ref{lem:2}]
  By hypothesis, there is an $\omega\ge 0$ such that
  $\varphi_{\omega}$ is convex, lower semicontinuous, and for every
  $c\in \R$, the sublevel sets 
  $\{\hat{u}\in V\,\vert\, \varphi_{\omega}(u) \le c\}$ are weakly
  relatively compact and closed. By Corollary~\ref{cor:27}, there is a lower
  semicontinuous, proper function
  $\varphi^{H} : H\to (-\infty,+\infty]$ such that
  $\varphi^{H}_{\omega}$ is convex and
  $\partial
  \varphi^{H}_{\omega}=\partial_{j}\varphi_{\omega}$. Applying
  Theorem~\ref{thm:29} to $\varphi_{\omega}$ and
  $\varphi^{H}_{\omega}$, we have that
\begin{equation}
  \label{eq:7}
    \varphi^{H}_{\omega}(u) = d+\inf_{\hat{u}\in
      j^{-1}(\{u\})}\varphi_{\omega}(\hat{u})\qquad\text{for every $u\in H$}
  \end{equation}
   and some constant $d\in \R$. For $c\in \R$, let $(u_{n})_{n\ge 1}$
   be an arbitrary sequence in $E_{\omega\!,c}$. 
   By~\eqref{eq:7}, for every $n\in \N$, there is a
   $\hat{u}_{n}\in j^{-1}(\{u_{n}\})$ such that
   \begin{displaymath}
     d+\varphi_{\omega}(\hat{u}_{n})\le c+ 1. 
   \end{displaymath}
   By hypothesis, all sublevel sets of $\varphi_{\omega}$ are weakly
   relatively compact in $V$. Thus, by our hypothesis, the image under
   $j$ is relatively compact in $H$. Consequently, there are a
   subsequence $(u_{n_{l}})_{l\ge 1}$ of $(u_{n})_{n\ge 1}$ and a
   $u\in H$ such that $u_{n_{l}}=j(\hat{u}_{n_{l}})\to u$ in $H$ as
   $l\to+\infty$. Since $\varphi^{H}_{\omega}(u_{n_{l}})\le c$ and
   since $\varphi^{H}$ is lower semicontinuous, it follows that
   $\varphi^{H}(u)\le c$. This shows that $E_{\omega\!,c}$ is compact.
\end{proof}

Now, applying Lemma~\ref{lem:2} to Theorem~\ref{thm:3}, we can state
the following existence theorem.

\begin{theorem}
  \label{thm:4}
  Let $\varphi : V\to (-\infty,+\infty]$ be proper, lower
  semicontinuous $j$-semiconvex, and $j$-elliptic. Assume that the
  mapping $j$ satisfies~\eqref{eq:25} and let $G : \He\to \He$ be
  a continuous mapping of sublinear growth~\eqref{eq:26}. Then, for
  $A=\partial_{j}\varphi$ the nonlinear evolution problem~\eqref{eq:6}
  admits for every $u_{0}\in \overline{j(D(\varphi))}$ and
  $f\in \He$ at least one solution
  $u\in H^{1}_{loc}((0,T];H)\cap C([0,T];H)$. In particular, one has that
  $\varphi\circ u$ belongs to $W^{1,1}_{loc}((0,T])\cap L^{1}(0,T)$ and
  inequality~\eqref{eq:14} holds. If
  $u_{0}\in D(\varphi)$, then problem~\eqref{eq:6} has a solution
  $u\in H^{1}(0,T;H)$. 
\end{theorem}

We complete this section by considering the following evolution problem
involving the \emph{Dirichlet-to-Neumann operator} associated with the
$p$-Laplacian (cf~\cite{MR3369257,MR3465809}).

\begin{example}
  Let $\Omega$ be a bounded domain with a Lipschitz continuous
  boundary $\partial\Omega$. Then, for $\frac{2d}{d+1}<p<+\infty$, the
  trace operator $\textrm{Tr} : W^{1,p}(\Omega)\to
  L^{2}(\partial\Omega)$ is a completely continuous operator
  (cf~\cite[Th\'eor\`eme~6.2]{MR0227584} for the case $p<d$, the other
  cases $p=d$ and $p>d$ can be deduced from \cite[Cons\'equence~6.2
  \&~6.3]{MR0227584}). Now, we take
  \begin{center}
    $V=W^{1,p}(\Omega)$, $H=L^{2}(\partial\Omega)$, and $j=\textrm{Tr}$.
  \end{center}
  Then, $j$ is a linear bounded mapping satisfying
  hypothesis~\eqref{eq:25}. In fact, $j$ is a prototype of a non-injective
  mapping. Furthermore, let $\varphi : V\to \R$ be the function given by
  \begin{displaymath}
    \varphi(\hat{u})=\tfrac{1}{p}\int_{\Omega}\abs{\nabla
      \hat{u}}^{p}\,\dx\qquad\text{for every $\hat{u}\in V$.}
  \end{displaymath}
  Then, $\varphi$ is continuously differentiable on $V$ and convex.
  Thus, the $\textrm{Tr}$-subdiffer\-ential operator
  $\partial_{\textrm{Tr}}\varphi$ is given by
  \begin{displaymath}
    \partial_{\textrm{Tr}}\varphi  = \Bigg\{(u,f)\in H\times H \;\Bigg\vert\;
      \begin{array}[c]{c}
        \exists \hat{u}\in V
        \text{ s.t. } \textrm{Tr}(\hat{u}) = u \text{ and for every }
        \hat{v}\in V\\[0,1cm]
        \int_{\Omega}\abs{\nabla \hat{u}}^{p-2}\nabla\hat{u}\nabla\hat{v}\,\dx=(f,j(\hat{v}))_H
      \end{array}
      \Bigg\}.
  \end{displaymath}
  Moreover, by inequality~\cite[(20)]{MR3369257}, for any $\omega>0$,
  the shifted function $\varphi_{\omega}$ has bounded level sets in
  $V$. Since $V$ is reflexive, every level set of $\varphi_{\omega}$
  is weakly compact in $V$. In addition,
  by~\cite[Lemma~2.1]{MR3369257}, $j(D(\varphi))$ is dense in $H$.
  
  Now, let $g : (0,T)\times \Omega\times \R\to \R$ be a
  Carath\'edory function with sublinear growth. Then by
  Theorem~\ref{thm:4}, for every $u_{0}\in L^{2}(\partial\Omega)$, there is at least one
  solution $u\in H^{1}_{loc}((0,T]; L^{2}(\partial\Omega))\cap C([0,T];L^{2}(\partial\Omega))$ of the
  elliptic-parabolic boundary-value problem
  \begin{displaymath}
    \begin{cases}
      \hspace{3,5cm}-\Delta_{p}\hat{u}(t,\cdot)=0 &\text{on
        $(0,T)\times \Omega$,}\\
      \partial_{t}u(t,\cdot)+ \abs{\nabla
        u(t,\cdot)}^{p-2}\tfrac{\partial}{\partial \nu}u(t,\cdot)=g(t,\cdot,u(t,\cdot)) & \text{on
      $(0,T)\times \partial\Omega$,}\\
    \hspace{2.15cm}\phantom{-\Delta_{p}u(t,\cdot)+}u(t,\cdot)=\hat{u}(t,\cdot) & \text{on
      $(0,T)\times \partial\Omega$,}\\
    \hspace{2,35cm}\phantom{-\Delta_{p}^{\! D}u(t,\cdot)}u(0,\cdot)=u_{0} &\text{on
      $\partial\Omega$.}
    \end{cases}
  \end{displaymath}
\end{example}

%
%

\appendix

\section{Br\'ezis' maximal $L^{2}$-regularity theorem}

To keep this paper self-contained, we show in this appendix that Br\'ezis' maximal
$L^2$-regularity result (Theorem~\ref{thm:2}) remains true for proper,
lower semicontinuous functions $\varphi : H\to (-\infty,+\infty]$,
which are \emph{semiconvex}.\medskip

Under the above hypotheses on $\varphi$, the subdifferential operator $A=\partial\varphi$ is
\emph{quasi maximal monotone}. Note that an operator $A$ on $H$
is called maximal monotone if firstly, $A$ is \emph{monotone}, that is,
\begin{displaymath}
  (v_{1}-v_{2},u_{1}-u_{2})_{H}\ge 0\qquad\text{for all
    $(u_{1},v_{1})$, $(u_{2},v_{2})\in A$,}
\end{displaymath}
and secondly, $A$ satisfies the \emph{range condition}
\begin{displaymath}
  \textrm{Rg}(I_{H}+\lambda A)=H\qquad\text{for one (or, equivalently for
    all) $\lambda>0$.}
\end{displaymath}
Now, an operator $A$ is called
\emph{quasi maximal monotone} if there is and $\omega\in \R$ such that $\omega
I_{H}+A$ is maximal monotone.

One important property of the class of maximal monotone operators in Hilbert
spaces is that their graph is closed in $H\times H_{w}$, where
$H_{w}$ means that $H$ is equipped with the weak topology $\sigma(H^{\ast},H)$.

\begin{proposition}[{\cite[Proposition~2.5]{MR0348562}}]
  \label{prop:A1}
  Let $A$ be an maximal monotone operator, $((u_{n},v_{n}))_{n\ge 1}\subseteq A$, $u$,
  $v\in H$ such that $u_{n}\rightharpoonup u$ and
  $v_{n}\rightharpoonup v$ weakly in $H$ as $n\to+\infty$ and
  $\limsup_{n\to+\infty}(u_{n},v_{n})_{H}\le (u,v)_{H}$. Then
  $(u,v)\in A$ and $(u_{n},v_{n})_{H}\to (u,v)_{H}$ as $n\to +\infty$.
\end{proposition}

For the class of $\omega$-quasi maximal monotone operators in Hilbert
spaces the following existence and regularity result holds. Here, we
recall~\cite[Theorem~4.5]{MR2582280} in the Hilbert spaces framework
and note that in Hilbert spaces \emph{monotone operators} are
\emph{accretive} and vice versa.

\begin{theorem}[{\bfseries Existence \& regularity for smooth $f$}]
  \label{thm:A1}
  Let $A$ be an $\omega$-quasi maximal monotone operator for some
  $\omega\in \R$, $f\in W^{1,1}(0,T;H)$, $u_{0}\in D(A)$. Then there
  is a unique solution $u\in W^{1,\infty}(0,T;H)$ of problem~\eqref{eq:5}.
\end{theorem}

Further, since $\partial\varphi_{\omega}=A+\omega I_{H}$ has dense
domain in $\overline{D(\varphi)}$ by~\cite[Proposition~2.11]{MR0348562} (or
\cite[p.48]{MR2582280}), the domain $D(A)$ of the subdifferential operator
$A=\partial\varphi$ is dense in $\overline{D(\varphi)}$. For later
use, we fix this observation in the next proposition.

\begin{proposition}
  \label{prop:A2}
  Let $\varphi : H\to (-\infty,+\infty]$ be proper, semiconvex, and
  lower semicontinuous. Then the domain $D(A)$ of $A=\partial\varphi$ is dense
  in $\overline{D(\varphi)}$.
\end{proposition}


We also need the following chain rule for convex functions $\varphi$.

\begin{lemma}[{\cite[Lemma~3.3]{MR0348562}}]
  \label{lem:A1}
  Let $\varphi : H\to (-\infty,+\infty]$ be proper, convex, and
  lower semicontinuous, and $u\in H^{1}(0,T;H)$.
  Assume, there is a $g\in \He$ such that $(u(t),g(t))\in \partial\varphi$ for
  a.e. $t\in (0,T)$. Then $\varphi\circ u$ is absolutely continuous on
  $[0,T]$ and 
  \begin{displaymath}
    \frac{\td}{\dt}\varphi(u(t))=(g(t),u(t))_{H}\qquad
    \text{
      for a.e. $t\in (0,T)$.}
  \end{displaymath}
\end{lemma}

Note, we may always assume
without loss of generality that $0\in D(\partial\varphi_{\omega})$,
$\varphi_{\omega}$ attains a minimum at $0$ (that is, \eqref{eq:3}
holds), and $\varphi_{\omega}(0)=0$. Otherwise, one chooses any
$(u_{0},v_{0})\in \partial\varphi_{\omega}$ and replaces $\varphi$ by
\begin{displaymath}
  \tilde{\varphi}(u):=\varphi(u+u_{0})-\varphi_{\omega}(u_{0})-(v_{0}-\omega
  u_{0},u)_{H}
  \qquad\text{for every $u\in H$.}
\end{displaymath}
Then,
\begin{displaymath}
  \tilde{\varphi}_{\omega}(u)=\varphi_{\omega}(u+u_{0})-\varphi_{\omega}(u_{0})-(v_{0},u)_{H}
  \qquad\text{for every $u\in H$,}
\end{displaymath}
$\varphi_{\omega}\ge 0$, $0\in D(\tilde{\varphi})$, and
$\tilde{\varphi}_{\omega}(0)=0$. Moreover, for each solution $y$ of inclusion
\begin{displaymath}
  \dot{y}(t)+\partial\tilde{\varphi}(y(t))\ni f(t)-v_{0}+\omega
  u_{0}\qquad\text{on $(0,T)$,}
\end{displaymath}
the function $u(t):=y(t)+u_{0}$ is a solution of~\eqref{eq:8}. This
shows that there is no loss of generality by assuming that for
$\varphi_{\omega}$, inequality~\eqref{eq:3} holds and
$\varphi_{\omega}\ge 0$.\medskip

With this, we can now outline the proof of Br\'ezis' $L^2$-maximal
regularity result. 

\begin{proof}[Proof of Theorem~\ref{thm:2}]
  Let $f\in \He$, $u_{0}\in D(\varphi)$, $f_{n}\in
  H^{1}(0,T;H)$ such that $f_{n}\to f$ in $\He$. Moreover, for every
  $n\ge 1$, there are $u_{n}^{(0)}\in D(A)$ such that 
  \begin{equation}
    \label{eq:18}
    \varphi_{\omega}(u_{n}^{(0)})\le \varphi_{\omega}(u_{0})
  \end{equation}
  and $u_{n}^{(0)}\to u_{0}$ in $H$ (see the last paragraph on~\cite[p.161]{MR2582280}). By
  Theorem~\ref{thm:A1}, there is a unique solution
  $u_{n}\in W^{1,\infty}(0,T;H)$ of problem
  \begin{displaymath}
   \begin{cases}
    \dot{u}_{n}+Au_{n}\ni f_{n} & \text{on $(0,T)$,}\\
    \hspace{17pt}u_{n}(0)=u_{n}^{(0)}.& 
   \end{cases}
  \end{displaymath}
 Then, by Lemma~\ref{lem:1}, $(u_{n})_{n\ge 1}$ is a Cauchy sequence
 in $C([0,T];H)$. Hence there is a $u\in C([0,T];H)$ such that
 $u_{n}\to u$ in $C([0,T];H)$. In particular, $u(0)=u_{0}$.\medskip

 (a) We show that $u$ satisfies~\eqref{eq:14}.
 Adding $\omega u_{n}$ on both sides of
 \begin{equation}
   \label{eq:17}
   \dot{u}_{n}+Au_{n}\ni f_{n}
 \end{equation}
 and then multiplying the resulting inclusion by $u_{n}$ yields
 \begin{displaymath}
   \tfrac{\td}{\dt}\tfrac{1}{2}\norm{u_{n}(t)}_{H}^{2}+(h+\omega
   u_{n}(t),u_{n}(t))_{H}=(f_{n}(t)+\omega u_{n}(t),u_{n}(t))_{H}
 \end{displaymath}
 for every $h\in Au_{n}(t)$ for a.e. $t\in
 (0,T)$. Applying~\eqref{eq:3}, and then integrating over $(0,t)$, for
 $t\in (0,T]$ leads to
 \begin{displaymath}
   \tfrac{1}{2}\norm{u_{n}(t)}_{H}^{2}\le
   \tfrac{1}{2}\norm{u_{n}^{(0)}}_{H}^{2}+\int_{0}^{t}\tfrac{1}{2}\norm{f_{n}(s)}_{H}^{2}\ds
   +(1+2\omega)\, \int_{0}^{t}\tfrac{1}{2}\norm{u_{n}(s)}_{H}^{2}\ds.
 \end{displaymath}
Now, the Gronwall inequality gives that $u_{n}$ satisfies the uniform
 bound~\eqref{eq:13} and by letting $n\to +\infty$ using that
 $u_{n}\to u$ in $C([0,T];H)$, we have that $u$
 satisfies~\eqref{eq:13}.\medskip

 (b) Next, we show that $u\in H^{1}(0,T;H)$. First, we add $\omega u_{n}$ on
 both sides of~\eqref{eq:17}, and then multiply the resulting
 inclusion by $\dot{u}_{n}$. Now, by Lemma~\ref{lem:A1},
 \begin{displaymath}
   \norm{\dot{u}_{n}(t)}_{H}^{2}+\tfrac{\td}{\dt}\varphi_{\omega}(u_{n}(t))=
   (f_{n}(t)+\omega u_{n}(t),\dot{u}_{n}(t))_{H}
 \end{displaymath}
 for a.e. $t\in (0,T)$. From this and by~\eqref{eq:18}, one deduces that
 \begin{align*} 
   &\tfrac{1}{2}\int_{0}^{t}\norm{\dot{u}_{n}(s)}_{H}^{2}\,\ds+\varphi_{\omega}(u_{n}(t))\\
   &\hspace{2cm}\le \varphi_{\omega}(u_{0})+
   \tfrac{1}{2}\int_{0}^{T}\norm{f_{n}(s)}_{H}^{2}\ds+\tfrac{\omega}{2}\norm{u_{n}(t)}_{H}^{2}
   -\tfrac{\omega}{2}\norm{u_{n}^{(0)}}_{H}^{2}.
 \end{align*}
 Note that $\varphi_{\omega}$ is bounded from below by an affine
 function. Thus and by part~(a), $(\dot{u}_{n})_{n\ge 1}$ is bounded
 in $\He$. Since $\He$ is reflexive,
 $(\dot{u}_{n})_{n\ge 1}$ admits a weakly convergent subsequence in
 $\He$. From this, by the limit $u_{n}\to u$ in $C([0,T];H)$,
 we can conclude that $u\in H^{1}(0,T;H)$. Moreover, by the lower
 semicontinuity of $\varphi_{\omega}$, one see that $u$ satisfies
 \begin{displaymath} 
   \tfrac{1}{2}\int_{0}^{t}\norm{\dot{u}(s)}_{H}^{2}\,\ds+\varphi_{\omega}(u(t))\le \varphi_{\omega}(u_{0})+
   \tfrac{1}{2}\int_{0}^{t}\norm{f(s)}_{H}^{2}\ds+\tfrac{\omega}{2}\norm{u(t)}_{H}^{2}
   -\tfrac{\omega}{2}\norm{u_{0}}_{H}^{2}
 \end{displaymath}
for every $t\in (0,T]$, which is equivalent to
 \begin{displaymath}
   \tfrac{1}{2}\int_{0}^{t}\norm{\dot{u}(s)}_{H}^{2}\,\ds+\varphi(u(t))\le \varphi(u_{0})
   +\tfrac{1}{2}\int_{0}^{t}\norm{f(s)}_{H}^{2}\ds.
 \end{displaymath}

(c) We conclude showing that $u$ is a solution of the evolution
problem~\eqref{eq:8}. For this, we use the lifted operator
$\mathcal{A}$ in $\He$ given by
\begin{displaymath}
  \mathcal{A}=\Big\{(u,v)\in \He\times\He\,\Big\vert\,v(t)\in
  Au(t)\text{ for a.e. }t\in (0,T)\Big\}.
\end{displaymath}
Since $\omega I_{H}+A=\partial\varphi_{\omega}$ is maximal monotone on $H$, we have that
$\mathcal{A}_{\omega}:=\omega\mathcal{I}_{\He}+\mathcal{A}$ is maximal monotone on
$\He$ (see~\cite[Exemple~2.3.3]{MR0348562}). Moreover, $u_{n}\to u$ in
$\He$, and after having chosen a subsequence,
$v_{n}:=f_{n}+\omega u_{n}-\dot{u}_{n}\rightharpoonup v:=f+\omega
u-\dot{u}$ weakly in $\He$. Thus, by Proposition~\ref{prop:A1}, $u\in
D(\mathcal{A})$ and $v\in \mathcal{A}_{\omega}u$, this is equivalent
to $u(t)\in D(A)$ and $f(t)-\dot{u}(t)\in Au(t)$ for a.e. $t\in
(0,T)$.\medskip

(d) Next, let $f\in \He$, $u_{0}\in \overline{D(\varphi)}$,
and $u_{n}^{(0)}\in D(\varphi)$ such that $u_{n}^{(0)}\to u_{0}$ in
$H$. By the previous part, for every $n\ge 1$, there are solutions
$u_{n}\in H^{1}(0,T;H)$ of problem
\begin{displaymath}
   \begin{cases}
    \dot{u}_{n}+Au_{n}\ni f & \text{on $(0,T)$,}\\
    \hspace{17pt}u_{n}(0)=u_{n}^{(0)}.& 
   \end{cases}
  \end{displaymath}
By Lemma~\ref{lem:1}, $(u_{n})_{n\ge 1}$ is a Cauchy sequence in
$C([0,T];H)$ and so, there is a $u\in C([0,T];H)$ such that $u_{n}\to
u$ in $C([0,T];H)$ as $n\to+\infty$. Moreover, by the same argument as in
part~(a), one sees that each $u_{n}$ and $u$
satisfies~\eqref{eq:13}.\medskip 

(e) Next, we show that 
\begin{equation}
  \label{eq:19}
  \begin{split}
  \int_{0}^{T}\varphi(u_{n}(s))\,\ds&\le
    \tfrac{1}{2}\norm{f}^{2}_{\He}+\tfrac{1+\omega}{2}
    \norm{u_{n}}^{2}_{\He} +\tfrac{1}{2}\norm{u_{n}^{(0)}}^{2}_{H}.
  \end{split}
\end{equation}
Since $f(t)-\dot{u}_{n}(t)\in \partial\varphi(u_{n}(t))$, it follows from
the definition of $A=\partial\varphi$ that
\begin{displaymath}
  \varphi_{\omega}(v)-\varphi_{\omega}(u_{n}(t))\ge ((f(t)-\dot{u}_{n}(t))+\omega
    u_{n}(t),v-u_{n}(t))_{H}
\end{displaymath}
for every $v\in H$ and a.e. $t\in (0,T)$. Thus taking $v=0$ and using
that $\varphi_{\omega}\ge 0$, one sees that
\begin{align*}
   0 \le \varphi_{\omega}(u_{n}(t))&\le -((f(t)-\dot{u}_{n}(t))+\omega
    u_{n}(t),-u_{n}(t))_{H}\\
    &=(f(t),u_{n}(t))_{H}- (\dot{u}_{n}(t),u_{n}(t))_{H}+\omega
    \norm{u_{n}(t)}_{H}^{2}\\
    &\le\tfrac{1}{2}\norm{f(t)}^{2}_{H}+(\tfrac{1}{2}+\omega)\norm{u_{n}(t)}^{2}_{H}
      -\tfrac{\td}{\dt}\tfrac{1}{2}\norm{u_{n}(t)}^{2}_{H}
  \end{align*}
  for a.e. $t\in (0,T)$. Integrating over
  $(0,T)$, one sees that
  \begin{equation}
    \label{eq:20}
    \begin{split}
      0\le \int_{0}^{T}\varphi_{\omega}(u_{n}(s))\,\ds&\le
      \tfrac{1}{2}\norm{f}^{2}_{\He} +\tfrac{1+2\omega}{2}\norm{u_{n}}^{2}_{\He} \\
      &\qquad\qquad
      -\tfrac{1}{2}\norm{u_{n}(t)}^{2}_{H}+\tfrac{1}{2}\norm{u_{n}^{(0)}}^{2}_{H}.
    \end{split}
  \end{equation}
  From this, it follows that \eqref{eq:19} holds. Then, since $u_{n}\to u$ in $C([0,T];H)$ and
$\varphi_{\omega}(u_{n})\ge 0$, it follows from the lower
semicontinuity of $\varphi_{\omega}$ and by Fatou's lemma
that~\eqref{eq:20} holds for $u$ and hence, $\varphi\circ u\in L^{1}(0,T)$
satisfying~\eqref{eq:21}.\medskip

(f) We show that $u\in H^{1}_{loc}((0,T];H)$ with $\sqrt{\cdot }\dot{u}\in \He$, and there is a
subsequence $(u_{n_{k}})_{k\ge 1}$ of $(u_{n})_{n\ge 1}$ such that
$\dot{u}_{n_{k}} \rightharpoonup \dot{u}$ weakly in $L^{2}_{loc}((0,T];H)$.
We first add $\omega u_{n}$ on both sides of
 \begin{displaymath}
   \dot{u}_{n}(t)+Au_{n}(t)\ni f(t), 
 \end{displaymath}
and then multiply the resulting inclusion by $t\cdot\dot{u}_{n}(t)$. Then
by Lemma~\ref{lem:A1},
 \begin{displaymath}
   \norm{\sqrt{t}\dot{u}_{n}(t)}_{H}^{2}+t\,\tfrac{\td}{\dt}\varphi_{\omega}(u_{n}(t))= t
   (f(t)+\omega u_{n}(t),\dot{u}_{n}(t))_{H}
 \end{displaymath}
 for a.e. $t\in (0,T)$. Applying Cauchy-Schwarz's and Young's
 inequality on the right hand side of this equation, and subsequently integrating
 over $(0,t)$ for $t\in (0,T]$ gives
 \begin{align*} 
   &\tfrac{1}{2}\int_{0}^{t}\norm{\sqrt{s}\dot{u}_{n}(s)}_{H}^{2}\ds+t\varphi_{\omega}(u_{n}(t))
   +\int_{0}^{t}\tfrac{\omega}{2}\norm{u_{n}(s)}_{H}^{2}\,\ds\\
   &\hspace{2cm}\le \int_{0}^{t}\varphi_{\omega}(u_{n}(s))\,\ds+
   \tfrac{1}{2}\int_{0}^{t}s\norm{f(s)}_{H}^{2}\ds+t\tfrac{\omega}{2}\norm{u_{n}(t)}_{H}^{2}.
 \end{align*}
Further, by~\eqref{eq:20} applied to $T=t$, one has
\begin{align*}
  &\tfrac{1}{2}\int_{0}^{t}\norm{\sqrt{s}\dot{u}_{n}(s)}_{H}^{2}\ds+t\varphi_{\omega}(u_{n}(t))
   +\int_{0}^{t}\tfrac{\omega}{2}\norm{u_{n}(s)}_{H}^{2}\,\ds\\
   &\hspace{2cm}\le \tfrac{1}{2}\norm{f}^{2}_{L^{2}(0,t;H))}
     +\tfrac{1+2\omega}{2}\norm{u_{n}}^{2}_{L^{2}(0,t;H)}
     -\tfrac{1}{2}\norm{u_{n}(t)}^{2}_{H}+\tfrac{1}{2}\norm{u_{n}^{(0)}}^{2}_{H}\\
      &\hspace{4cm} 
        +\tfrac{1}{2}\norm{\sqrt{\cdot}f}_{L^{2}(0,t;H)}^{2}+t\tfrac{\omega}{2}\norm{u_{n}(t)}_{H}^{2}.
\end{align*}
 Recall that 
 $u_{n}\to u$ in $C([0,T];H)$. Thus, $(\sqrt{\cdot}\dot{u}_{n})_{n\ge 1}$ is bounded
 in $\He$ and so by the reflexivity of $\He$, one has
 that $u\in H^{1}_{loc}((0,T];H)$ with $\sqrt{\cdot}\dot{u}\in \He$.
 In particular, $(\dot{u}_{n})_{n\ge 1}$ is bounded in
 $L^{2}(\delta,T;H)$ for every $\delta\in (0,T]$. Thus, a diagonal
 sequence arguments shows that there is a subsequence
 $(u_{n_{k}})_{k\ge 1}$ of $(u_{n})_{n\ge 1}$ such that
 $\dot{u}_{n_{k}} \rightharpoonup \dot{u}$ weakly in
 $L^{2}_{loc}((0,T];H)$.\medskip

(g) Next, we show that $u$ is a solutions
of~\eqref{eq:5} and $\varphi\circ u\in W^{1,1}_{loc}((0,T])$. To see
that $u$ is a solution of~\eqref{eq:5} recall that $u_{n}^{(0)}\to
u_{0}$ in $H$ and the solutions $u_{n}$ of~\eqref{eq:17} converge to
$u$ in $C([0,T];H)$. Thus, $u(0)=u_{0}$ and since for every $\delta\in
(0,T]$, $\dot{u}_{n}\rightharpoonup \dot{u}$ weakly in
$L^{2}(\delta,T;H)$, it follows by the same argument as in part~(c)
from the maximal monotonicity of the
operator $\omega I_{\mathcal{H}}+\mathcal{A}_{\delta}$ in
$L^{2}(\delta,T;H)$ with
\begin{displaymath}
  \mathcal{A}_{\delta}=\Big\{(u,v)\in L^{2}(\delta,T;H)\times L^{2}(\delta,T;H)\,\Big\vert\,v(t)\in
  Au(t)\text{ for a.e. }t\in (\delta,T)\Big\}.
\end{displaymath}
that $u(t)\in D(A)$ for a.e. $t\in (0,T)$ and
$f(t)-\dot{u}(t)\in A(u(t))$.  Moreover, since now,
$g(t):=f(t)-\dot{u}(t)+\omega u(t)\in \partial\varphi_{\omega}(u(t))$
for a.e. $t\in (0,T)$ and $g\in L^{2}(\delta,T;H)$ for every
$\delta\in (0,T]$, it follows from Lemma~\ref{lem:A1} that
$\varphi\circ u\in W^{1,1}_{loc}((0,T])$. This completes the proof of
Br\'ezis' $L^{2}$-maximal regularity result for semiconvex
$\varphi$.\medskip

(h) Finally, we show that $u$ satisfies~\eqref{eq:10}
and~\eqref{eq:11}. Since $u$ is a solution of~\eqref{eq:5}, we can add
$\omega u$ on both side of 
\begin{displaymath}
   \dot{u}(t)+Au(t)\ni f(t), 
 \end{displaymath}
and then multiply the resulting inclusion by
$t\cdot\dot{u}(t)$. Recall, $\sqrt{\cdot}\dot{u}\in
\He$. Thus by Lemma~\ref{lem:A1},
 \begin{displaymath}
   \norm{\sqrt{t}\dot{u}(t)}_{H}^{2}+t\,\tfrac{\td}{\dt}\varphi_{\omega}(u(t))= t
   (f(t)+\omega u(t),\dot{u}(t))_{H}
 \end{displaymath}
 for a.e. $t\in (0,T)$. Next, by Cauchy-Schwarz's and Young's
 inequality, and subsequently integrating
 over $(0,t)$ for $t\in (0,T]$ gives
 \begin{align*} 
   &\tfrac{1}{2}\int_{0}^{t}\norm{\sqrt{s}\dot{u}(s)}_{H}^{2}\ds+t\varphi_{\omega}(u(t))
   +\int_{0}^{t}\tfrac{\omega}{2}\norm{u(s)}_{H}^{2}\,\ds\\
   &\hspace{2cm}\le \int_{0}^{t}\varphi_{\omega}(u(s))\,\ds+
   \tfrac{1}{2}\int_{0}^{t}s\norm{f(s)}_{H}^{2}\ds+t\tfrac{\omega}{2}\norm{u(t)}_{H}^{2},
 \end{align*}
from which we can conclude~\eqref{eq:10}
and~\eqref{eq:11}.
\end{proof}

%
%


\begin{thebibliography}{10}

\bibitem{MR2103696}
{\sc W.~Arendt}, {\em Semigroups and evolution equations: functional calculus,
  regularity and kernel estimates}, in Evolutionary equations. {V}ol. {I},
  Handb. Differ. Equ., North-Holland, Amsterdam, 2004, pp.~1--85.

\bibitem{MR2722654}
{\sc W.~Arendt and R.~Chill}, {\em Global existence for quasilinear diffusion
  equations in isotropic nondivergence form}, Ann. Sc. Norm. Super. Pisa Cl.
  Sci. (5), 9 (2010), pp.~523--539.

\bibitem{MR2823661}
{\sc W.~Arendt and A.~F.~M. ter Elst}, {\em The {D}irichlet-to-{N}eumann
  operator on rough domains}, J. Differential Equations, 251 (2011),
  pp.~2100--2124.

\bibitem{MR2881534}
\leavevmode\vrule height 2pt depth -1.6pt width 23pt, {\em Sectorial forms and
  degenerate differential operators}, J. Operator Theory, 67 (2012),
  pp.~33--72.

\bibitem{MR2582280}
{\sc V.~Barbu}, {\em Nonlinear differential equations of monotone types in
  {B}anach spaces}, Springer Monographs in Mathematics, Springer, New York,
  2010.

\bibitem{MR3403408}
{\sc Z.~Belhachmi and R.~Chill}, {\em Application of the {$j$}-subgradient in a
  problem of electropermeabilization}, J. Elliptic Parabol. Equ., 1 (2015),
  pp.~13--29.

\bibitem{arXiv:1804.08272}
{\sc Z.~Belhachmi and R.~Chill}, {\em The bidomain problem as a gradient
  system}, ArXiv e-prints,  (2018).

\bibitem{Benilanbook}
{\sc P.~B{\'e}nilan, M.~G. Crandall, and A.~Pazy}, {\em Evolution problems
  governed by accretive operators}, book in preparation.

\bibitem{MR0283635}
{\sc H.~Br\'{e}zis}, {\em Propri\'{e}t\'{e}s r\'{e}gularisantes de certains
  semi-groupes non lin\'{e}aires}, Israel J. Math., 9 (1971), pp.~513--534.

\bibitem{MR0348562}
\leavevmode\vrule height 2pt depth -1.6pt width 23pt, {\em Op\'{e}rateurs
  maximaux monotones et semi-groupes de contractions dans les espaces de
  {H}ilbert}, North-Holland Publishing Co., Amsterdam-London; American Elsevier
  Publishing Co., Inc., New York, 1973.
\newblock North-Holland Mathematics Studies, No. 5. Notas de Matem\'{a}tica
  (50).

\bibitem{MR2759829}
{\sc H.~Br\'{e}zis}, {\em Functional analysis, {S}obolev spaces and partial
  differential equations}, Universitext, Springer, New York, 2011.

\bibitem{MR3465809}
{\sc R.~Chill, D.~Hauer, and J.~Kennedy}, {\em Nonlinear semigroups generated
  by {$j$}-elliptic functionals}, J. Math. Pures Appl. (9), 105 (2016),
  pp.~415--450.

\bibitem{CoulHau2017}
{\sc T.~Coulhon and D.~Hauer}, {\em Regularisation effects of nonlinear
  semigroups - theory and applications}.
\newblock to appear in BCAM Springer Briefs, 2017.

\bibitem{MR2597943}
{\sc L.~C. Evans}, {\em Partial differential equations}, vol.~19 of Graduate
  Studies in Mathematics, American Mathematical Society, Providence, RI,
  second~ed., 2010.

\bibitem{MR3369257}
{\sc D.~Hauer}, {\em The {$p$}-{D}irichlet-to-{N}eumann operator with
  applications to elliptic and parabolic problems}, J. Differential Equations,
  259 (2015), pp.~3615--3655.

\bibitem{HauMaz2019}
{\sc D.~Hauer and J.~M. Maz\'on}, {\em The {D}irichlet-to-{N}eumann operator
  associated with the $1$-{L}aplacian}.
\newblock submitted, 2019.

\bibitem{MR0227584}
{\sc J.~Ne{\v{c}}as}, {\em Les m\'ethodes directes en th\'eorie des \'equations
  elliptiques}, Masson et Cie, \'Editeurs, Paris, 1967.

\bibitem{MR0342978}
{\sc H.~H. Schaefer}, {\em Topological vector spaces}, Springer-Verlag, New
  York-Berlin, 1971.
\newblock Third printing corrected, Graduate Texts in Mathematics, Vol. 3.

\bibitem{MR1033498}
{\sc E.~Zeidler}, {\em Nonlinear functional analysis and its applications.
  {II}/{B}}, Springer-Verlag, New York, 1990.
\newblock Nonlinear monotone operators, Translated from the German by the
  author and Leo F. Boron.

\end{thebibliography}

\end{document}